\documentclass[amstex,12pt, amssymb]{article}

\usepackage{mathtext}
\usepackage[cp1251]{inputenc}
\usepackage[T2A]{fontenc}
\usepackage[dvips]{graphicx}
\usepackage{amsmath}
\usepackage{amssymb}
\usepackage{amsxtra}
\usepackage{latexsym}
\usepackage{ifthen}

\newcounter{lemma}[section]

\newcounter{corollary}[section]

\newcounter{remark}[section]

\newcounter{theorem}[section]

\newcounter{proposition}[section]

\newcounter{example}

\numberwithin{equation}{section}

\DeclareMathOperator*{\esssup}{ess\,sup}

\begin{document}

\markboth{M. MATELJEVI\'{C}, E.~SEVOST'YANOV}{\centerline{ON MONTEL
THEOREM ...}}

\def\cc{\setcounter{equation}{0}
\setcounter{figure}{0}\setcounter{table}{0}}

\overfullrule=0pt


\author{MIODRAG MATELJEVI\'{C}, EVGENY SEVOST'YANOV}

\title{
{\bf ON MONTEL THEOREM FOR MAPPINGS WITH INVERSE MODULI
INEQUALITIES}}

\date{\today}
\maketitle

\begin{abstract}
This paper is devoted to the study of mappings with finite
distortion, in particular, mappings satisfying the inverse Poletskii
inequality. We study the problem of equicontinuity of families of
such mappings in a given domain. We establish that a family of open
discrete mappings with the inverse Poletskii inequality, omitting at
least one point, is equicontinuous if the majorant responsible for
the distortion of the modulus of families of paths under the mapping
is integrable over almost all concentric spheres centered at the
given point. Since analytic functions with finite multiplicity
satisfy the inverse Poletskii inequality, this result generalizes
the well-known Montel theorem on the normality of families.
\end{abstract}


\section{Introduction}

This paper is devoted to the study of mappings with bounded and
finite distortion, see, e.g., \cite{AM}, \cite{BV},
\cite{Cr$_1$}--\cite{Cr$_3$}, \cite{MRV},
\cite{MRSY$_1$}--\cite{MRSY$_2$}, \cite{M$_1$}--\cite{M$_4$},
\cite{MSS}, \cite{PSS}, \cite{RV}, \cite{SalSt} and~\cite{Va}. The
well-known Montel theorem states that, under certain (minimal)
conditions, a family of analytic functions on a plane is normal
(equicontinuous). In particular, the following result holds.

\medskip
{\bf Theorem~A.}(\cite[$\S$ 32, Ch.~II]{Mont}). {\it\, A family
$\frak{F}_{a, b}(D)$ of all analytic functions $f:D\rightarrow {\Bbb
C}\setminus\{a, b\}$ of a domain $D\subset {\Bbb C}$ is normal for
any fixed $a, b\in {\Bbb C},$ $a\ne b.$}

\medskip
For quasiconformal mappings this result looks like this.

\medskip
{\bf Theorem~B.} (\cite[Theorem~19.2]{Va}). {\,\it A family ${\cal
F}_r,$ consisting of all $K$-quasiconformal mappings $f:D\rightarrow
\overline{{\Bbb R}^n}\setminus\{a_f, b_f\}$ is equicontinuous
whenever $h(a_f, b_f)\geqslant r>0$ (where $h$ -- is a chordal
metric and $r>0$ does not depend on $f$).}

\medskip
Finally, in the theory of quasiregular spatial mappings the
following result is known.

\medskip
{\bf Theorem~C.} (\cite[Corollary~3.14.IV]{Ri}). {\it\, Let $G$ be a
domain in $\overline{{\Bbb R}^n},$ $n\geqslant 3,$ and let
$K\geqslant 1.$ There exists a number $q_0,$ depending only $n$ and
$K,$ such that the following is true: any family ${\mathcal F}$ of
$K$-quasimeromorphic mappings $f:G\rightarrow \overline{{\Bbb R}^n}$
such that $f$ omits distinct points $a^f_1,\ldots, a^f_{q_0}$ is
equicontinuous whenever there exists $\gamma>0$ such that
$\sigma^f_0=\frac{1}{4}\min\limits_{j\ne k}\sigma(a^f_j,
a^f_k)\geqslant \gamma$ for any $f\in {\mathcal F},$ where
$\sigma(\cdot,\cdot)$ denotes the spherical distance between points
in $\overline{{\Bbb R}^n}.$}

\medskip
As we see, all the results of {\bf Theorems A-C} require, in one way
or another, some conditions on the omitting points in a given family
of mappings. Our immediate goal is to obtain analogs of these
theorems for a broader class of mappings. Let us turn to
definitions.

\medskip
A Borel function $\rho:{\Bbb R}^n\,\rightarrow [0,\infty] $ is
called {\it admissible} for the family $\Gamma$ of paths $\gamma$ in
${\Bbb R}^n,$ if the relation
\begin{equation}\label{eq1.4}
\int\limits_{\gamma}\rho (x)\, |dx|\geqslant 1
\end{equation}
holds for all (locally rectifiable) paths $ \gamma \in \Gamma.$ In
this case, we write: $\rho \in {\rm adm} \,\Gamma .$ Let $p\geqslant
1,$ then {\it $p$-modulus} of $\Gamma $ is defined by the equality
\begin{equation}\label{eq1.3gl0}
M_p(\Gamma)=\inf\limits_{\rho \in \,{\rm adm}\,\Gamma}
\int\limits_{{\Bbb R}^n} \rho^p (x)\,dm(x)\,.
\end{equation}
We set $M(\Gamma):=M_n(\Gamma).$ Let $y_0\in {\Bbb R}^n,$
$0<r_1<r_2<\infty$ and
\begin{equation}\label{eq1**}
A(y_0, r_1,r_2)=\left\{ y\,\in\,{\Bbb R}^n:
r_1<|y-y_0|<r_2\right\}\,.\end{equation}
Given sets $E,$ $F\subset\overline{{\Bbb R}^n}$ and a domain
$D\subset {\Bbb R}^n$ we denote by $\Gamma(E,F,D)$ the family of all
paths $\gamma:[a,b]\rightarrow \overline{{\Bbb R}^n}$ such that
$\gamma(a)\in E,\gamma(b)\in\,F$ and $\gamma(t)\in D$ for $t \in (a,
b).$ If $f:D\rightarrow {\Bbb R}^n$ is a given mapping, $y_0\in
\overline{f(D)}$ and $0<r_1<r_2<d_0=\sup\limits_{y\in f(D)}|y-y_0|,$
then by $\Gamma_f(y_0, r_1, r_2)$ we denote the family of all paths
$\gamma$ in $D$ such that $f(\gamma)\in \Gamma(S(y_0, r_1), S(y_0,
r_2), A(y_0,r_1,r_2)).$ Let $Q:{\Bbb R}^n\rightarrow [0, \infty]$ is
a Lebesgue measurable function. We say that {\it $f$ satisfies
inverse Poletsky inequality relative to $p$-modulus} at the point
$y_0\in \overline{f(D)},$ $p\geqslant 1,$ if there is
$r_0=r_0(y_0)>0,$ $r_0<d_0=\sup\limits_{y\in f(D)}|y-y_0|,$ such
that the ratio
\begin{equation}\label{eq2*A}
M_p(\Gamma_f(y_0, r_1, r_2))\leqslant \int\limits_{f(D)\cap
A(y_0,r_1,r_2)} Q(y)\cdot \eta^p (|y-y_0|)\, dm(y)
\end{equation}
holds for any $0<r_1<r_2<r_0$ and any Lebesgue measurable function
$\eta: (r_1,r_2)\rightarrow [0,\infty ]$ such that
\begin{equation}\label{eqA2}
\int\limits_{r_1}^{r_2}\eta(r)\, dr\geqslant 1\,.
\end{equation}

\begin{remark}\label{rem1}
Here and below
\begin{equation}\label{eq23}
N(y, f, D)\,=\,{\rm card}\,\left\{x\in D: f(x)=y\right\}\,.
\end{equation}
Note that all quasiregular mappings $f:D\rightarrow {\Bbb R}^n$
satisfy the condition
\begin{equation}\label{eq22}
M(\Gamma_f(y_0, r_1, r_2))\leqslant \int\limits_{f(D)\cap
A(y_0,r_1,r_2)} K_O\cdot N(y, f, D)\cdot \eta^n (|y-y_0|)\, dm(y)
\end{equation}
at each point $y_0\in {\Bbb R}^n\setminus\{\infty\}$ with
$K_O=K_O(f)=\esssup\limits_{x\in D}\frac{\Vert
f^{\,\prime}(x)\Vert^n}{|J(x, f)|}\geqslant 1$ and an arbitrary
Lebesgue-dimensional function $\eta: (r_1,r_2)\rightarrow
[0,\infty],$ which satisfies condition~(\ref{eqA2}). Indeed,
quasiregular mappings satisfy the condition
\begin{equation}\label{eq24}
M(\Gamma_f(y_0, r_1, r_2))\leqslant \int\limits_{f(D)\cap
A(y_0,r_1,r_2)} K_O\cdot N(y, f, D)\cdot (\rho^{\,\prime})^n(y)\,
dm(y)
\end{equation}
for an arbitrary function $\rho^{\,\prime}\in{\rm adm}\,
f(\Gamma_f(y_0, r_1, r_2)),$ see \cite[Remark~2.5.II]{Ri}. Put
$\rho^{\,\prime}(y):=\eta(|y-y_0|)$ for $y\in A(y_0,r_1,r_2)\cap
f(D),$ and $\rho^{\,\prime}(y)=0$ otherwise. By Luzin theorem, we
may assume that the function $\rho^{\,\prime}$ is Borel measurable
(see, e.g., \cite[Section~2.3.6]{Fe}). Due
to~\cite[Theorem~5.7]{Va},
$$\int\limits_{\gamma_*}\rho^{\,\prime}(y)\,|dy|\geqslant
\int\limits_{r_1}^{r_2}\eta(r)\,dr\geqslant 1$$
for each (locally rectifiable) path $\gamma_*$ in $\Gamma(S(y_0,
r_1), S(y_0, r_2), A(y_0, r_1, r_2)).$ By substituting the function
$\rho^{\,\prime}$ mentioned above into~(\ref{eq24}), we obtain the
desired ratio~(\ref{eq22}).
\end{remark}

\medskip
Recall that a mapping $f:D\rightarrow {\Bbb R}^n$ is called {\it
discrete} if the pre-image $\{f^{-1}\left(y\right)\}$ of each point
$y\,\in\,{\Bbb R}^n$ consists of isolated points, and {\it is open}
if the image of any open set $U\subset D$ is an open set in ${\Bbb
R}^n.$ Later, in the extended space $\overline{{{\Bbb R}}^n}={{\Bbb
R}}^n\cup\{\infty\}$ we use the {\it spherical (chordal) metric}
$h(x,y)=|\pi(x)-\pi(y)|,$ where $\pi$ is a stereographic projection
of $\overline{{{\Bbb R}}^n}$ onto the sphere
$S^n(\frac{1}{2}e_{n+1},\frac{1}{2})$ in ${{\Bbb R}}^{n+1},$ namely,
\begin{equation}\label{eq3C}
h(x,\infty)=\frac{1}{\sqrt{1+{|x|}^2}}\,,\qquad
h(x,y)=\frac{|x-y|}{\sqrt{1+{|x|}^2} \sqrt{1+{|y|}^2}}\,, \quad x\ne
\infty\ne y
\end{equation}
(see \cite[Definition~12.1]{Va}). Further, the closure
$\overline{A}$ and the boundary $\partial A$ of the set $A\subset
\overline{{\Bbb R}^n}$ we understand relative to the chordal metric
$h$ in $\overline{{\Bbb R}^n}.$ Earlier, the second author obtained
some results regarding the equicontinuity of families of mappings
in~(\ref{eq2*A}), see, e.g. \cite{SevSkv$_1$}--\cite{SevSkv$_3$} and
\cite{SSD}. Now, we give one of the most general versions on this
topic.

\medskip
For domains $D, D^{\,\prime}\subset {\Bbb R}^n,$ $n\geqslant 2,$ and
a Lebesgue measurable function $Q:{\Bbb R}^n\rightarrow [0, \infty]$
equal to zero outside the domain $D^{\,\prime},$ we define by
$\frak{R}_Q(D, D^{\,\prime})$ the family of all open discrete
mappings $f:D\rightarrow D^{\,\prime}$ such that
relation~(\ref{eq2*A}) holds for each point $y_0\in D^{\,\prime}.$
Note that the definition of class $\frak{R}_Q(D, D^{\,\prime})$ does
not require that the domain $D$ be mapped onto the domain
$D^{\,\prime}$ surjectively under the mapping $f\in \frak{R}_Q(D,
D^{\,\prime}).$ In what follows, $\mathcal{H}^{n-1}$ denotes
$(n-1)$-dimensional Hausdorff measure. The following result holds.

\medskip
{\bf Theorem~D.} (\cite[Theorem~1.1]{SevSkv$_3$};
\cite[Theorem~1.5]{SevSkv$_2$}). {\it Let $D$ and $D^{\,\prime}$ be
domains in ${\Bbb R}^n,$ $n\geqslant 2,$ and let $D^{\,\prime}$ be a
bounded domain. Suppose that, for each point $y_0\in D^{\,\prime}$
and for every $0<r_1<r_2<r_0:=\sup\limits_{y\in
D^{\,\prime}}|y-y_0|$ there is a set $E\subset[r_1, r_2]$ of a
positive linear Lebesgue measure such that the function $Q$ is
integrable with respect to $\mathcal{H}^{n-1}$ over the spheres
$S(y_0, r)$ for every $r\in E.$ Then the family of mappings
$\frak{R}_Q(D, D^{\,\prime})$ is equicontinuous at each point
$x_0\in D.$ In particular, the latter holds, if the assumption on
$Q$ mentioned above replace by the simpler one: $Q\in L_{\rm
loc}^1(D^{\,\prime}).$}

\medskip
Note that, {\bf Theorem~D} generalized {\bf Theorems~A-C} for the
case, when the mapped domain $D^{\,\prime}$ is bounded and
multiplicity function $N(y, f, D)$ is uniformly bounded. Indeed, by
Remark~\ref{rem1} any quasiregular mapping mentioned above
satisfies~(\ref{eq2*A})--(\ref{eqA2}). At the same time, this
theorem is obviously not a complete analogue of any of {\bf
Theorems~A-C}, since its conditions imply restrictions on the mapped
domain and $N(y, f, D).$ On the other hand,
in~\cite{SevSkv$_1$}--\cite{SevSkv$_3$} and \cite{SSD}, we also
established similar results, which do not assume any restrictions on
the mapped domain, but for the case of an integrable~$Q$. These
results also do not give a complete picture, since the problem of
the existence of such mappings may arise whenever the mapped domain
is unbounded. The result established below brings further clarity to
the study of this problem. Here, we require that the family of
mappings omit at least one point, and that the majorant $Q$ be only
integrable by the spheres in the neighborhood of some (one) point.

\medskip
Given a domain $D\subset {\Bbb R}^n,$ $n\geqslant 2,$ a Lebesgue
measurable function $Q:{\Bbb R}^n\rightarrow [0, \infty]$ and a
point $a\in {\Bbb R}^n$ we denote by $\frak{F}_{Q, a}(D)$ a family
of all open discrete mappings $f:D\rightarrow {\Bbb
R}^n\setminus\{a\}$ satisfying the
relations~(\ref{eq2*A})--(\ref{eqA2}) at all finite points $y_0\in
{\Bbb R}^n.$ The following statement holds.

\medskip
\begin{theorem}\label{th2}
{\it\, Assume that, the following condition holds: for each point
$y_0\in {\Bbb R}^n$ there exists $r_0=r_0(y_0)>0$ such that, for
every $0<r_1<r_2<r_0$ there is a set $E\subset[r_1, r_2]$ of a
positive linear Lebesgue measure such that the function $Q$ is
integrable with respect to $\mathcal{H}^{n-1}$ over the spheres
$S(y_0, r)$ for every $r\in E.$ Then the family of mappings
$\frak{F}_{Q, a}(D)$ is equicontinuous at each point $x_0\in D.$ }
\end{theorem}

\medskip
\begin{corollary}\label{cor1}
{\it\, The statement of Theorem~\ref{th2} is true if the condition
on the function~$Q$ mentioned above is replaced by a simpler one:
$Q\in L_{\rm loc}^{1}({\Bbb R}^n).$}
\end{corollary}

\medskip
Theorem~\ref{th2} applies to the case where the family of mappings
omits at least one point. However, an analog of this statement may
be given when the presence of excluded points is not necessary;
furthermore, the function $Q$ may be assumed to be integrable in the
neighborhood of just one point.

\medskip
Let $\frak{F}$ be a family of mappings $f:D\rightarrow {\Bbb R}^n.$
We set
\begin{equation*} C(x_0, \frak{F}):=\{y\in \overline{{\Bbb
R}^n}:\exists\,x_k\in D, f_k\in \frak{F}: x_k\rightarrow x_0,
f_k(x_k) \rightarrow y, k\rightarrow\infty\}\,.
\end{equation*}
Given a domain $D\subset {\Bbb R}^n,$ $n\geqslant 2,$ and a Lebesgue
measurable function $Q:{\Bbb R}^n\rightarrow [0, \infty]$ we denote
by $\frak{R}_{Q}(D)$ some family of all open discrete mappings
$f:D\rightarrow {\Bbb R}^n$ satisfying the
relations~(\ref{eq2*A})--(\ref{eqA2}) at all finite points $y_0\in
{\Bbb R}^n.$ The following statement holds.

\medskip
\begin{theorem}\label{th1}
{\it\, Let $x_0\in D.$ Assume that, there exists $y_0\in C(x_0,
\frak{R}_{Q}(D))$ such that:

\medskip
1) there exists $r_0=r_0(y_0)>0$ such that, for every
$0<r_1<r_2<r_0$ there is a set $E\subset[r_1, r_2]$ of a positive
linear Lebesgue measure such that the function $Q$ is integrable
with respect to $\mathcal{H}^{n-1}$ over the spheres $S(y_0, r)$ for
every $r\in E;$

\medskip
2) there exists a sequence $y_m\in B(y_0, r_0),$ $y_m\ne y_0,$
$y_m\rightarrow y_0$ as $m\rightarrow\infty$ such that, for every
$m\in {\Bbb N}$ there exists $C_m>0$ such that $d(x_0, A_m)\geqslant
C_m,$ where
$$A_m:=\{x\in D: \exists\, f\in \frak{R}_{Q}(D): f(x)=y_m\}\,.$$
Then the family of mappings $\frak{R}_{Q}(D)$ is equicontinuous at
$x_0\in D.$ }
\end{theorem}

\medskip
\begin{corollary}\label{cor2}
{\it\, The statement of Theorem~\ref{th1} is true if the
condition~1) mentioned in this theorem is replaced by a simpler one:
$Q\in L_{\rm loc}^{1}({\Bbb R}^n).$}
\end{corollary}

\section{The case where the family of mappings omits at least two points}

Let $D\subset {\Bbb R}^n,$ $f:D\rightarrow {\Bbb R}^n$ be a discrete
open mapping, $\beta: [a,\,b)\rightarrow {\Bbb R}^n$ be a path, and
$x\in\,f^{\,-1}(\beta(a)).$ A path $\alpha: [a,\,c)\rightarrow D$ is
called a {\it maximal $f$-lifting} of $\beta$ starting at $x,$ if
$(1)\quad \alpha(a)=x\,;$ $(2)\quad f\circ\alpha=\beta|_{[a,\,c)};$
$(3)$\quad for $c<c^{\prime}\leqslant b,$ there is no a path
$\alpha^{\prime}: [a,\,c^{\prime})\rightarrow D$ such that
$\alpha=\alpha^{\prime}|_{[a,\,c)}$ and $f\circ
\alpha^{\,\prime}=\beta|_{[a,\,c^{\prime})}.$ Similarly, we may
define a maximal $f$-lifting $\alpha: (c,\,b]\rightarrow D$ of a
path $\beta: (a,\,b]\rightarrow {\Bbb R}^n$ ending at
$x\in\,f^{\,-1}(\beta(b)).$ The following assertion holds
(see~\cite[Lemma~3.12]{MRV}).

\medskip
\begin{proposition}\label{pr3}
{\it Let $f:D\rightarrow {\Bbb R}^n,$ $n\geqslant 2,$ be an open
discrete mapping, let $x_0\in D,$ and let $\beta: [a,\,b)\rightarrow
{\Bbb R}^n$ be a path such that $\beta(a)=f(x_0)$ and such that
either $\lim\limits_{t\rightarrow b}\beta(t)$ exists, or
$\beta(t)\rightarrow \partial f(D)$ as $t\rightarrow b.$ Then
$\beta$ has a maximal $f$-lifting $\alpha: [a,\,c)\rightarrow D$
starting at $x_0.$ If $\alpha(t)\rightarrow x_1\in D$ as
$t\rightarrow c,$ then $c=b$ and $f(x_1)=\lim\limits_{t\rightarrow
b}\beta(t).$ Otherwise $\alpha(t)\rightarrow \partial D$ as
$t\rightarrow c.$}
\end{proposition}

\medskip
Recall that a {\it path} will be called a continuous mapping
$\gamma: I\rightarrow {\Bbb R}^n$ of a segment, interval or
half-interval $I\subset {\Bbb R}$ into $n$-dimensional Euclidean
space ${\Bbb R}^n.$ As usual, the following set is called the {\it
locus} of a path $\gamma: I\rightarrow {\Bbb R}^n:$
$$|\gamma|=\{x\in {\Bbb R}^n: \exists\, t\in [a, b]:
\gamma(t)=x\}\,.$$
We also say that the paths $\gamma_1$ and $\gamma_2$ do not
intersect each other if their loci do not intersect as sets in
${\Bbb R}^n.$ A path $\gamma:I\rightarrow {\Bbb R}^n$ is called a
{\it Jordan arc,} if $\gamma$ is a homeomorphism of $I$ onto
$|\gamma|.$ The following lemma was proved in
\cite[Lemma~2.1]{SevSkv$_1$}.

\medskip
\begin{lemma}\label{lem1H}{\it\,
Let $n\geqslant 2, $ and let $D$ be a domain in ${\Bbb R}^n$ that is
locally connected on its boundary. Then every two pairs of points
$a\in D, b\in \overline{D}$ and $c\in D, d\in \overline{D}$ such
that $a\ne c$ and $b\ne d$ can be joined by non-intersecting paths
$\gamma_1:[0, 1]\rightarrow \overline{D}$ and $\gamma_2:[0,
1]\rightarrow \overline{D}$ so that $\gamma_i(t)\in D$ for all $t\in
(0, 1)$ and all $i=1,2,$ while $\gamma_1(0)=a,$ $\gamma_1(1)=b,$
$\gamma_2(0)=c$ and $\gamma_2(1)=d.$}
\end{lemma}

\medskip
The following statement is a simple consequence of the well-known
V\"{a}is\"{a}l\"{a} theorem on the lower estimate of the modulus of
families of paths joining two continua that intersect the plates of
a spherical ring, see e.g. \cite[Lemma~11.2]{Sev$_2$}.

\medskip
\begin{lemma}\label{lem2C} {\bf(V\"{a}is\"{a}l\"{a}'s lemma on
the weak flatness of inner points).} {\,\it Let $n\geqslant 2 $, let
$D$ be a domain in $\overline{{\Bbb R}^n},$ and let $x_0\in D.$ Then
for each $P>0$ and each neighborhood $U$ of point $x_0$ there is a
neighborhood $V\subset U$ of the same point such that $M(\Gamma(E,
F, D))> P$ for any continua $E, F \subset D $ intersecting $\partial
U$ and $\partial V.$}
\end{lemma}

\medskip
Given a domain $D\subset {\Bbb R}^n,$ $n\geqslant 2,$ a Lebesgue
measurable function $Q:{\Bbb R}^n\rightarrow [0, \infty]$ and points
$a, b\in {\Bbb R}^n,$ $a\ne b,$ we denote by $\frak{F}_{Q, a}(D)$ a
family of all open discrete mappings $f:D\rightarrow {\Bbb
R}^n\setminus\{a, b\}$ satisfying the
relations~(\ref{eq2*A})--(\ref{eqA2}) at every point $y_0\in {\Bbb
R}^n.$ The following statement holds.

\medskip
\begin{theorem}\label{th3}
{\it\, Assume that, the following condition holds: for each point
$y_0\in {\Bbb R}^n$ there exists $r_0=r_0(y_0)>0$ such that, for
every $0<r_1<r_2<r_0$ there is a set $E\subset[r_1, r_2]$ of a
positive linear Lebesgue measure such that the function $Q$ is
integrable with respect to $\mathcal{H}^{n-1}$ over the spheres
$S(y_0, r)$ for every $r\in E.$ Then the family of mappings
$\frak{F}_{Q, a, b}(D)$ is equicontinuous at each point $x_0\in D.$
}
\end{theorem}

\medskip
\begin{proof} The proof makes substantial use of the scheme we applied
earlier for the case of uniformly bounded families of mappings,
see~\cite[Theorem~3]{Sev$_1$}, cf.~\cite[Theorem~1.1]{SevSkv$_3$}.
Let us prove Theorem~\ref{th3} by the contradiction. Suppose that
the family $\frak{F}_{Q, a, b}(D)$ is not equicontinuous at some
point $x_0\in D.$ Then there is $\varepsilon_0>0,$ for which the
following condition is true: for any $m\in{\Bbb N}$ there is $x_m
\in D$ with $|x_m-x_0|<1/m,$ and a mapping $f_m \in \frak{F}_{Q, a,
b}(D)$ such that
\begin{equation}\label{eq13A}
h(f_m(x_m), f_m(x_0))\geqslant \varepsilon_0.
 \end{equation}
We may consider that $D$ is bounded. Since $\overline{{\Bbb R}^n}$
is a compact space, we may consider that sequences $f_m(x_m)$ and
$f_m(x_0)$ converge to $\overline{x_1}$ and $\overline{x_2}\in
\overline{{\Bbb R}^n}$ as $m\rightarrow\infty,$ respectively.
By~(\ref{eq13A}) and by the continuity of the metrics, we obtain
that $h(\overline{x_1}, \overline{x_2})\geqslant\varepsilon_0/2.$
Next we need to consider the following two fundamentally different
cases: {\bf Case~1: $\overline{x_1}\ne\infty\ne \overline{x_2};$}
{\bf Case~2: One of the points $\overline{x_1}$ or $\overline{x_2}$
is finite, and one is equal to infinity.}

\medskip
{\bf Case 1: $\overline{x_1}\ne\infty\ne \overline{x_2}.$} There are
also four major subcases possible here:

\medskip
{\bf Case 1.1: $a\ne \overline{x_1}\ne b, a\ne \overline{x_2}\ne
b;$}

\medskip
{\bf Case 1.2: exactly one of the points, $\overline{x_1}$ or
$\overline{x_2}$ coincides with either $a$ or $b;$}

\medskip
{\bf Case 1.3: both points $\overline{x_1}$ and $\overline{x_2}$
coincide with $a$ and $b.$}

\medskip
We should look at all three cases in detail.

\medskip
{\bf Case 1.1: $a\ne \overline{x_1}\ne b, a\ne \overline{x_2}\ne
b.$} By Lemma~\ref{lem1H} we may join points $a$ and
$\overline{x_1},$ as well as the points $b$ and $\overline{x_2}$ by
disjoint paths $\gamma_1 \colon [1/2, 1] \rightarrow {\Bbb R}^n$ and
$\gamma_2 \colon[1/2, 1] \rightarrow {\Bbb R}^n$ respectively. Let
$R_1>0$ be such that $\overline{B(\overline{x_1}, R_1)}\cap
|\gamma_2|=\varnothing,$ and let $R_2>0$ be such that
$$
(\overline{B(\overline{x_1},R_1)}\cup|\gamma_1|)\cap\overline{B(\overline{x_2},R_2)}=\varnothing.
 $$
We may assume that $f_m(x_m)\in B(\overline{x_1}, R_1)$ and
$f_m(x_0)\in B(\overline{x_2},R_2)$ for any $m\geqslant 1.$ Join the
points $f_m(x_m)$ and $\overline{x_1}$ by a path
$\alpha^{\,*}_m\colon [0, 1/2] \rightarrow B(\overline {x_1}, R_1),$
and join the point $f_m(x_0)$ with the point $\overline {x_2}$ by a
path $\beta^{\,*}_m\colon[0, 1/2]\rightarrow B(\overline{x_2}, R_2)$
(see Figure~\ref{figure2}).
 \begin{figure}
\centerline{\includegraphics[scale=0.5]{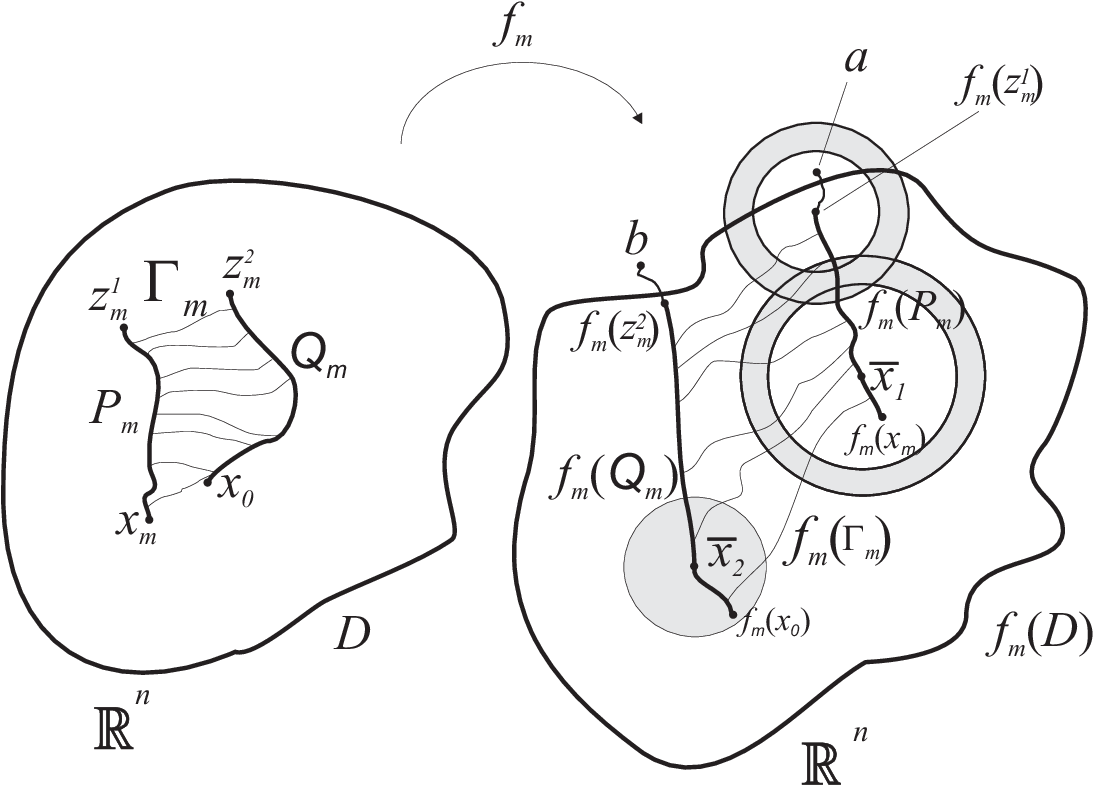}} \caption{To
the proof of Theorem~\ref{th3}, Case~1.1}\label{figure2}
 \end{figure} Set
$$
\alpha_m(t)=\quad\left\{\begin{array}{rr}
\alpha^*_m(t), & t\in [0, 1/2],\\
\gamma_1(t), & t\in [1/2, 1]\end{array} \right.\,,\quad
\beta_m(t)=\quad\left\{
\begin{array}{rr}
\beta^*_m(t), & t\in [0, 1/2],\\
\gamma_2(t), & t\in [1/2, 1].\end{array} \right.
 $$
By the construction, the sets
 $$
A_1:=|\gamma_1|\cup \overline{B(\overline{x_1}, R_1)}\,,\quad
A_2:=|\gamma_2|\cup \overline{B(\overline{x_2}, R_2)}
 $$
do not intersect, in particular, there is $\varepsilon^{\,*}_1>0$
such that
\begin{equation}\label{eq1J}
d(A_1,A_2)\geqslant \varepsilon^{\,*}_1>0\,.
 \end{equation}
Let $r_0=r_0(y)>0$ be the number from the conditions of the theorem,
defined for each $y\in{\Bbb R}^n.$ Set
$$r_*(y):=\min\{\varepsilon^{\,*}_1, r_0(y)\}\,.$$
Cover the set $A_1$ with balls $B(y, r_*/4),$ $y\in A_1.$ Note that
$|\gamma_1 |$ is a compact set in ${\Bbb R}^n$ as a continuous image
of the compact set $[1/2, 1]$ under the mapping $\gamma_1.$ Then, by
the Heine-Borel-Lebesgue lemma, there is a finite subcover
$\bigcup\limits_{i=1}^pB(y_i, r_*/4)$ of the set $A_1.$ In other
words,
\begin{equation}\label{eq2A}
A_1\subset \bigcup\limits_{i=1}^pB(y_i,r_i/4)\,,\qquad 1\leqslant
p<\infty\,,
 \end{equation}
where $r_i$ denotes $r_*(y_i)$ for any $y_i\in {\Bbb R}^n.$

\medskip
Let $\alpha^0_m:[0, c_1)\rightarrow D$ and $\beta^0_m:[0,
c_2)\rightarrow D$ be maximal $f_m$-liftings of paths $\alpha_m$ and
$\beta_m$ starting at points $x_m$ and $x_0,$ respectively. Such
maximal liftings exist by Proposition~\ref{pr3}. By the same
Proposition,  $\alpha^0_m(t_k)\rightarrow
\partial D$ and $\beta^0_m(t^{\,\prime}_k)\rightarrow
\partial D$ for some sequences $t_k\rightarrow c_1-0$
and $t^{\,\prime}_k\rightarrow c_2-0,$ $k\rightarrow\infty.$ Then
there are sequences of points $z^1_m\in |\alpha_m^0|$ and $z^2_m\in
|\beta_m^0|$ such that $d(z^1_m,
\partial D)<1/m$ and $d(z^2_m,\partial D)<1/m.$
Since $\overline{D}$ is a compactum, we may consider that
$z^1_m\rightarrow p_1\in
\partial D$ and $z^2_m\rightarrow p_2\in
\partial D$ as $m\rightarrow\infty,$ where $p_1\ne\infty\ne p_2.$
Let $P_m$ be the part of the locus of the path $\alpha^0_m$ in $D,$
located between the points $x_m$ and $z^1_m,$ and $Q_m$ the part of
the locus of the path $\beta^0_m$ in $D,$ located between the points
$x_0$ and $z^2_m.$ By the construction, $f_m(P_m)\subset A_1$ and
$f_m(Q_m)\subset A_2.$ Put $\Gamma_m:=\Gamma(P_m, Q_m, D).$ Recall
that we write $\Gamma_1>\Gamma_2$ if and only if each path
$\gamma_1\in \Gamma_1$ has a subpath $\gamma_2\in \Gamma_2.$ (In
other words, if $\gamma_1\colon I\rightarrow{\Bbb R}^n,$ then
$\gamma_2\colon J\rightarrow{\Bbb R}^n,$ where $J\subset I$ and
$\gamma_2(t)=\gamma_1(t)$ for $t\in J,$ $I, J$ are segments,
intervals, or half-intervals). Then, by~(\ref{eq1J}) and
(\ref{eq2A}), and by~\cite[Theorem~1.I.5.46]{Ku}, we obtain that
\begin{equation}\label{eq5}
\Gamma_m>\bigcup\limits_{i=1}^p\Gamma_{im}\,,
 \end{equation}
where $\Gamma_{im}:=\Gamma_{f_m}(y_i, r_i/4, r_i/2).$ Set
$\widetilde{Q}(y)=\max\{Q(y), 1\}$ and
$$\widetilde{q}_{y_i}(r)=\frac{1}{\omega_{n-1}r^{n-1}}\int\limits_{S(y_i,
r)}\widetilde{Q}(y)\,d\mathcal{A}\,,$$
where $\omega_{n-1}$ denotes the area of the unit sphere in ${\Bbb
R}^n,$ and $d\mathcal{A}=d\mathcal{H}^{n-1}$ is the area element on
$S(y_i, r).$
Now, we have also that $\widetilde{q}_{y_i}(r)\ne \infty$ for any
$r\in E\subset [r_i/4,r_i/2],$ where $E$ is some set of a positive
linear measure which exists by the assumption.
Set
$$I_i=I_i(y_i,r_i/4,r_i/2)=\int\limits_{r_i/4}^{r_i/2}\
\frac{dr}{r\widetilde{q}_{y_i}^{\frac{1}{n-1}}(r)}\,.$$
Observe that $I_i\ne 0,$ because $\widetilde{q}_{y_i}(r)\ne \infty$
for any $r\in E\subset [r_i/4,r_i/2].$ Besides that, note that
$I_i\ne\infty,$ because
$$I_i\leqslant \log\frac{r_2}{r_1}<\infty\,,\quad i=1,2, \ldots, p\,.$$
Now, we put
$$\eta_i(r)=\begin{cases}
\frac{1}{I_ir\widetilde{q}_{y_i}^{\frac{1}{n-1}}(r)}\,,&
r\in [r_i/4,r_i/2]\,,\\
0,& r\not\in [r_i/4,r_i/2]\,.
\end{cases}$$
Observe that, a function~$\eta_i$ satisfies the
condition~$\int\limits_{r_i/4}^{r_i/2}\eta_i(r)\,dr=1,$ therefore it
can be substituted into the right side of the
inequality~(\ref{eq2*A}) with the corresponding values $f,$ $r_1$
and $r_2.$ We will have that
\begin{equation}\label{eq7B}
M(\Gamma_{im})\leqslant \int\limits_{A(y_i, r_i/4, r_i/2)}
\widetilde{Q}(y)\,\eta^n_i(|y-y_i|)\,dm(y)\,.\end{equation}
On the other hand, by the Fubini theorem we obtain that
\begin{eqnarray}\label{eq7C}\int\limits_{A(y_i, r_i/4, r_i/2)}
\widetilde{Q}(y)\,\eta^n_i(|y-y_i|)\,dm(y)\nonumber=\\
= \int\limits_{r_i/4}^{r_i/2}\int\limits_{S(y_i,
r)}Q(y)\eta^n_i(|y-y_i|)\,d\mathcal{A}\,dr\,=\\
=\frac{\omega_{n-1}}{I_i^n}\int\limits_{r_i/4}^{r_i/2}r^{n-1}
\widetilde{q}_{y_i}(r)\cdot
\frac{dr}{r^n\widetilde{q}^{\frac{n}{n-1}}_{y_i}(r)}=\frac{\omega_{n-1}}{I_i^{n-1}}\,,\nonumber
\end{eqnarray}
where $\omega_{n-1}$ is the area of the unit sphere in ${\Bbb R}^n.$
Now, by~(\ref{eq7B}) and~(\ref{eq7C}) we obtain that
$$M(\Gamma_{im})\leqslant \frac{\omega_{n-1}}{I_i^{n-1}}\,,$$
whence from~(\ref{eq5})
\begin{equation}\label{eq7D}
M(\Gamma_m)\leqslant \sum\limits_{i=1}^pM(\Gamma_{im})\leqslant
\sum\limits_{i=1}^p\frac{\omega_{n-1}}{I_i^{n-1}}:=C_0\,, \quad
m=1,2,\ldots\,.
\end{equation}
Further reasoning is related to the ``weak flatness'' of the inner
points of the domain $D.$ Notice, that $d(P_m)\geqslant |x_m-
z^1_m|\geqslant (1/2)\cdot |x_0-p_1|>0$ and $d(Q_m)\geqslant |x_0-
z^2_m|\geqslant(1/2)\cdot |x_0-p_2|>0$ for sufficiently large $m\in
{\Bbb N},$ in addition,
$$
d(P_m, Q_m)\leqslant |x_m-x_0|\rightarrow 0, \quad m\rightarrow
\infty.
 $$
Now, by Lemma~\ref{lem2C}
\begin{equation}\label{eq1A}
M(\Gamma_m)=M(\Gamma(P_m, Q_m, D))\rightarrow\infty\,,\quad
m\rightarrow\infty,
\end{equation}
which contradicts the relation~(\ref{eq7D}). The resulting
contradiction indicates that the assumption in~(\ref{eq13A}) is
wrong, which completes the {\bf Case 1.1}.

\medskip
As we will see further, the consideration of the two remaining cases
is not significantly different from case~1.1.

\medskip
{\bf Case 1.2: exactly one of the points, $\overline{x_1}$ or
$\overline{x_2}$ coincides with either $a$ or $b.$} Without loss of
generality, we may assume that $a=\overline{x_1}.$ Now, $
\overline{x_1}\ne b\ne \overline{x_2}.$

\medskip
We join the points $b$ and $\overline{x_2}$ by a path $\gamma_2
\colon[1/2, 1] \rightarrow {\Bbb R}^n$ which does not intersect $a.$
Let $R_1>0$ be such that $\overline{B(\overline{x_1}, R_1)}\cap
|\gamma_2|=\varnothing,$ and let $R_2>0$ be such that
$$
(\overline{B(\overline{x_1},R_1)})\cap\overline{B(\overline{x_2},R_2)}=\varnothing.
 $$
We may assume that $f_m(x_m)\in B(\overline{x_1}, R_1)$ and
$f_m(x_0)\in B(\overline{x_2},R_2)$ for any $m\geqslant 1.$ Join the
points $f_m(x_m)$ and $\overline{x_1}=a$ by a path $\alpha_m\colon
[0, 1] \rightarrow B(\overline {x_1}, R_1),$ and join the point
$f_m(x_0)$ with the point $\overline {x_2}$ by a path
$\beta^{\,*}_m\colon[0, 1/2]\rightarrow B(\overline{x_2}, R_2)$ (see
Figure~\ref{figure3}).
 \begin{figure}
\centerline{\includegraphics[scale=0.5]{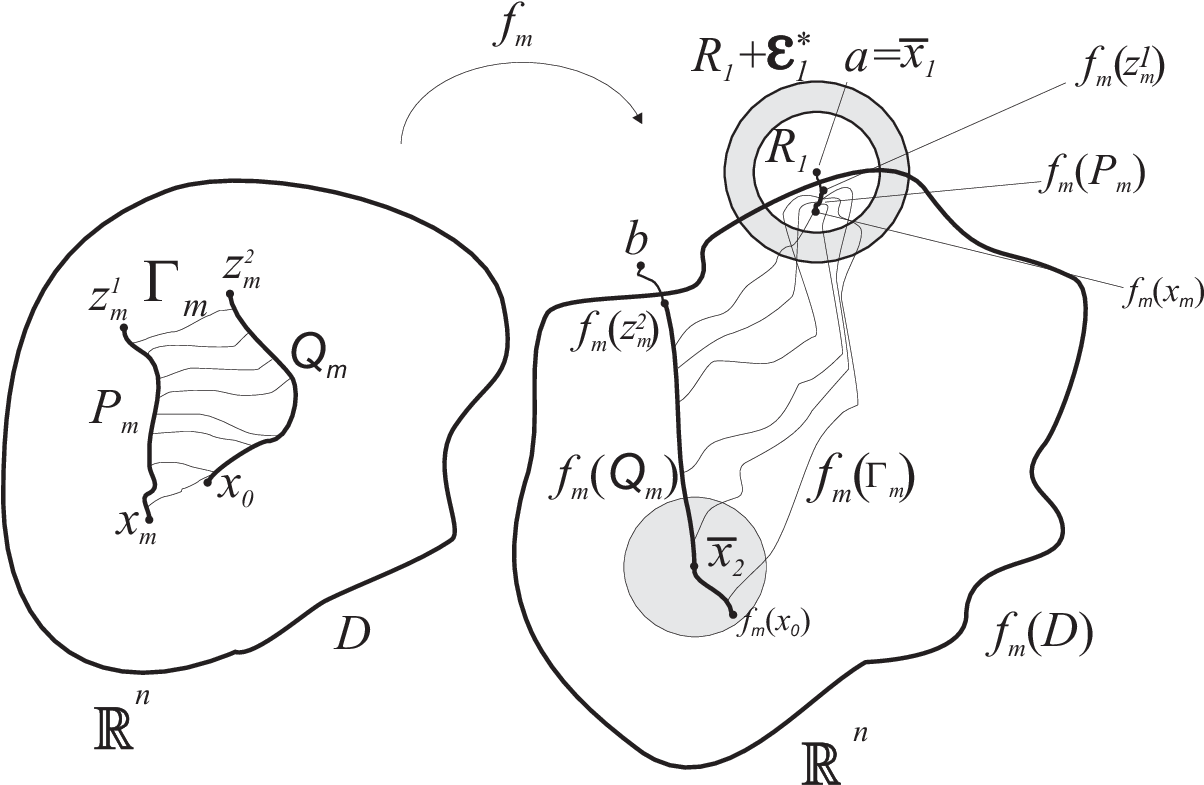}} \caption{To
the proof of Theorem~\ref{th3}, Case~1.2}\label{figure3}
 \end{figure} Set
$$
\beta_m(t)=\quad\left\{
\begin{array}{rr}
\beta^*_m(t), & t\in [0, 1/2],\\
\gamma_2(t), & t\in [1/2, 1].\end{array} \right.
 $$
By the construction, the sets
 $$
A_1:=\overline{B(\overline{x_1}, R_1)}\,,\quad A_2:=|\gamma_2|\cup
\overline{B(\overline{x_2}, R_2)}
 $$
do not intersect, in particular, there is $\varepsilon^{\,*}_1>0$
such that
\begin{equation}\label{eq1JA}
\overline{B(\overline{x_1}, R_1+\varepsilon^{\,*}_1)}\cap
A_2=\varnothing.
 \end{equation}
Let $\alpha^0_m:[0, c_1)\rightarrow D$ and $\beta^0_m:[0,
c_2)\rightarrow D$ be maximal $f_m$-liftings of paths $\alpha_m$ and
$\beta_m$ starting at points $x_m$ and $x_0,$ respectively (they
exist by Proposition~\ref{pr3}). By the same Proposition,
$\alpha^0_m(t_k)\rightarrow
\partial D$ and $\beta^0_m(t^{\,\prime}_k)\rightarrow
\partial D$ for some sequences $t_k\rightarrow c_1-0$
and $t^{\,\prime}_k\rightarrow c_2-0,$ $k\rightarrow\infty.$ Then
there are sequences of points $z^1_m\in |\alpha_m^0|$ and $z^2_m\in
|\beta_m^0|$ such that $d(z^1_m,
\partial D)<1/m$ and $d(z^2_m,\partial D)<1/m.$
Since $\overline{D}$ is a compactum, we may consider that
$z^1_m\rightarrow p_1\in
\partial D$ and $z^2_m\rightarrow p_2\in
\partial D$ as $m\rightarrow\infty.$
Let $P_m$ be the part of the locus of the path $\alpha^0_m$ in $D,$
located between the points $x_m$ and $z^1_m,$ and $Q_m$ the part of
the locus of the path $\beta^0_m$ in $D,$ located between the points
$x_0$ and $z^2_m.$ By the construction, $f_m(P_m)\subset A_1$ and
$f_m(Q_m)\subset A_2.$ Put $\Gamma_m:=\Gamma(P_m, Q_m, D).$ Then,
by~\cite[Theorem~1.I.5.46]{Ku}, we obtain that
\begin{equation}\label{eq5A}
\Gamma_m>\Gamma_0\,,
 \end{equation}
where $\Gamma_{0}:=\Gamma_{f_m}(a, R_1, R_1+\varepsilon^{\,*}_1).$
Let $\widetilde{Q}(y)=\max\{Q(y), 1\}$ and let
$$\widetilde{q}_{a}(r)=\frac{1}{\omega_{n-1}r^{n-1}}\int\limits_{S(a,
r)}\widetilde{Q}(y)\,d\mathcal{A}\,.$$
Similarly to the case~1.1 we may show that $\widetilde{q}_{a}(r)\ne
\infty$ for any $r\in E\subset [R_1, R_1+\varepsilon^{\,*}_1].$
Set
$$I_0=I(a, R_1, R_1+\varepsilon^{\,*}_1)=
\int\limits_{R_1}^{R_1+\varepsilon^{\,*}_1}\
\frac{dr}{r\widetilde{q}_{a}^{\frac{1}{n-1}}(r)}\,.$$
As in the case~1.1, we may show that $0<I_0<\infty.$ Now, we put
$$\eta_0(r)=\begin{cases}
\frac{1}{I_0r\widetilde{q}_{a}^{\frac{1}{n-1}}(r)}\,,&
r\in [R_1, R_1+\varepsilon^{\,*}_1]\,,\\
0,& r\not\in [R_1, R_1+\varepsilon^{\,*}_1]\,.
\end{cases}$$
Observe that, a function~$\eta_0$ satisfies the
condition~$\int\limits_{R_1}^{R_1+\varepsilon^{\,*}_1}\eta_0(r)\,dr=1,$
therefore it can be substituted into the right side of the
inequality~(\ref{eq2*A}) with the corresponding values $f,$ $r_1$
and~$r_2.$ Arguing similarly to~(\ref{eq7C}) and
applying~(\ref{eq5A}), we will have that
\begin{equation}\label{eq7BA}
M(\Gamma_m)\leqslant M(\Gamma_0)\leqslant \int\limits_{A(a, R_1,
R_1+\varepsilon^{\,*}_1)}
\widetilde{Q}(y)\,\eta^n_0(|y-a|)\,dm(y)=\frac{\omega_{n-1}}{I_0^{n-1}}\,,\end{equation}
where $\omega_{n-1}$ is the area of the unit sphere in ${\Bbb R}^n.$
On the other hand, similarly to~(\ref{eq1A}),
$$
M(\Gamma_m)=M(\Gamma(P_m, Q_m, D))\rightarrow\infty\,,\quad
m\rightarrow\infty,
 $$
which contradicts the relation~(\ref{eq7BA}). This completes the
consideration of {\bf Case 1.2}.

\medskip
{\bf Case 1.3: both points $\overline{x_1}$ and $\overline{x_2}$
coincide with $a$ and $b.$} Let $a=\overline{x_1}$ and
$b=\overline{x_2}.$ Let $R_1>0$ and $R_2>0$ be such that
\begin{equation}\label{eq2B}
(\overline{B(\overline{x_1},R_1)})\cap\overline{B(\overline{x_2},R_2)}=\varnothing\,.
\end{equation}
We may assume that $f_m(x_m)\in B(\overline{x_1}, R_1)$ and
$f_m(x_0)\in B(\overline{x_2},R_2)$ for any $m\geqslant 1.$ Join the
points $f_m(x_m)$ and $\overline{x_1}=a$ by a path $\alpha_m\colon
[0, 1] \rightarrow B(\overline {x_1}, R_1),$ and join the point
$f_m(x_0)$ with the point $\overline {x_2}$ by a path
$\beta_m\colon[0, 1]\rightarrow B(\overline{x_2}, R_2)$ (see
Figure~\ref{figure4}).
 \begin{figure}
\centerline{\includegraphics[scale=0.5]{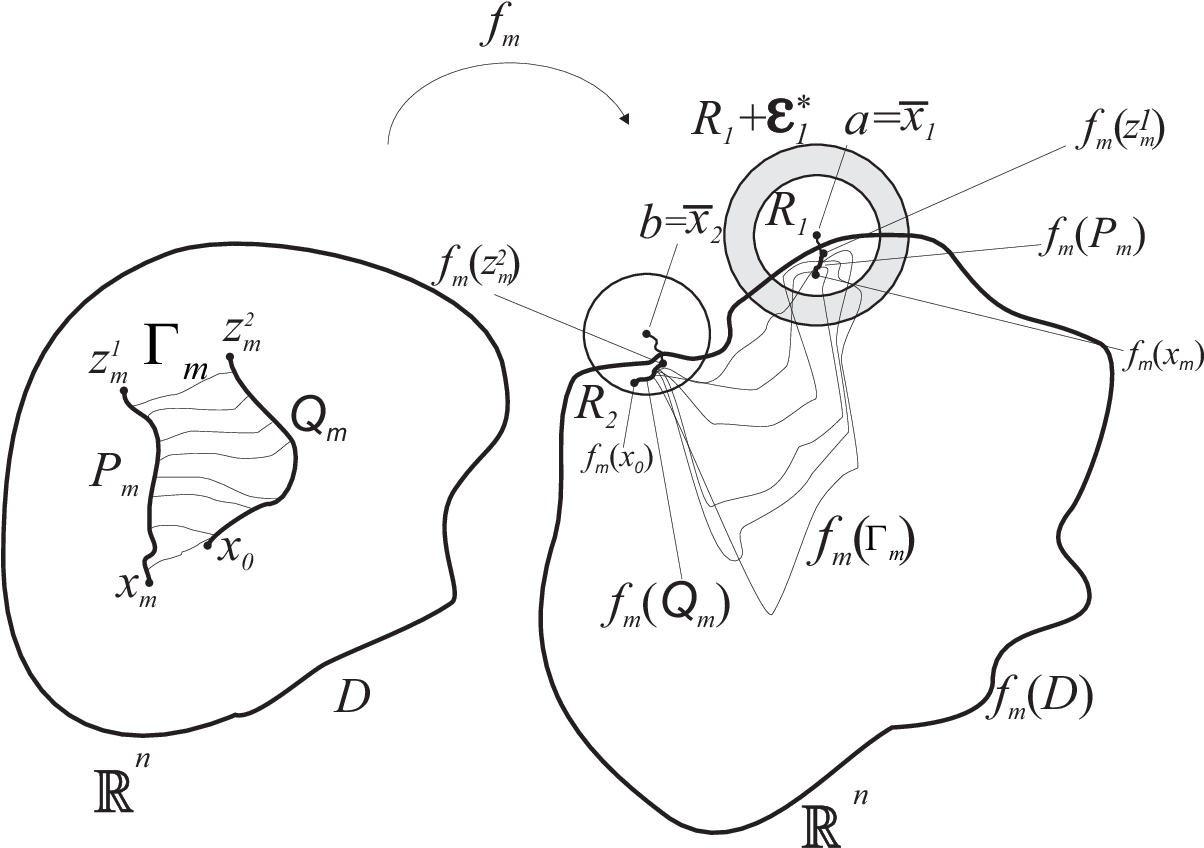}} \caption{To
the proof of Theorem~\ref{th3}, Case~1.3}\label{figure4}
 \end{figure}
Due to~(\ref{eq2B}), there is $\varepsilon^{\,*}_1>0$ such that
$$
\overline{B(\overline{x_1}, R_1+\varepsilon^{\,*}_1)}\cap
A_2=\varnothing.
$$
Let $\alpha^0_m:[0, c_1)\rightarrow D$ and $\beta^0_m:[0,
c_2)\rightarrow D$ be maximal $f_m$-liftings of paths $\alpha_m$ and
$\beta_m$ starting at points $x_m$ and $x_0,$ respectively (they
exist by Proposition~\ref{pr3}). By the same Proposition,
$\alpha^0_m(t_k)\rightarrow
\partial D$ and $\beta^0_m(t^{\,\prime}_k)\rightarrow
\partial D$ for some sequences $t_k\rightarrow c_1-0$
and $t^{\,\prime}_k\rightarrow c_2-0,$ $k\rightarrow\infty.$ Then
there are sequences of points $z^1_m\in |\alpha_m^0|$ and $z^2_m\in
|\beta_m^0|$ such that $d(z^1_m,
\partial D)<1/m$ and $d(z^2_m,\partial D)<1/m.$
Since $\overline{D}$ is a compactum, we may consider that
$z^1_m\rightarrow p_1\in
\partial D$ and $z^2_m\rightarrow p_2\in
\partial D$ as $m\rightarrow\infty.$
Let $P_m$ be the part of the locus of the path $\alpha^0_m$ in
${\Bbb R}^n,$ located between the points $x_m$ and $z^1_m,$ and
$Q_m$ the part of the locus of the path $\beta^0_m$ in ${\Bbb M}^n,$
located between the points $x_0$ and $z^2_m.$ By the construction,
$f_m(P_m)\subset B(\overline{x_1}, R_1)$ and $f_m(Q_m)\subset {\Bbb
R}^n\setminus  B(\overline{x_1}, R_1+\varepsilon^{\,*}_1)$ for some
$\varepsilon^{\,*}_1>0.$ Put $\Gamma_m:=\Gamma(P_m, Q_m, D).$ Then,
by~\cite[Theorem~1.I.5.46]{Ku}, we obtain that
\begin{equation}\label{eq5B}
\Gamma_m>\Gamma_0\,,
 \end{equation}
where $\Gamma_{0}:=\Gamma_{f_m}(a, R_1, R_1+\varepsilon^{\,*}_1).$
Let $\widetilde{Q}(y)=\max\{Q(y), 1\}$ and let
$$\widetilde{q}_{a}(r)=\frac{1}{\omega_{n-1}r^{n-1}}\int\limits_{S(a,
r)}\widetilde{Q}(y)\,d\mathcal{A}\,.$$
Similarly to the case~1.1 we may show that $\widetilde{q}_{a}(r)\ne
\infty$ for any $r\in E\subset [R_1, R_1+\varepsilon^{\,*}_1],$
where $E$ is some set of positive linear measure which exists by the
assumption.
Set
$$I_0=I_0(a, R_1, R_1+\varepsilon^{\,*}_1)=
\int\limits_{R_1}^{R_1+\varepsilon^{\,*}_1}\
\frac{dr}{r\widetilde{q}_{a}^{\frac{1}{n-1}}(r)}\,.$$
As in the case~1.1, we may show that $0<I_0<\infty.$ Now, we put
$$\eta_0(r)=\begin{cases}
\frac{1}{I_0r\widetilde{q}_{a}^{\frac{1}{n-1}}(r)}\,,&
r\in [R_1, R_1+\varepsilon^{\,*}_1]\,,\\
0,& r\not\in [R_1, R_1+\varepsilon^{\,*}_1]\,.
\end{cases}$$
Observe that, a function~$\eta_0$ satisfies the
condition~$\int\limits_{R_1}^{R_1+\varepsilon^{\,*}_1}\eta_0(r)\,dr=1,$
therefore it can be substituted into the right side of the
inequality~(\ref{eq2*A}) with the corresponding values $f,$ $r_1$
and $r_2.$ Arguing similarly to~(\ref{eq7C}) and
applying~(\ref{eq5B}), we will have that
\begin{equation}\label{eq7BB}
M(\Gamma_m)\leqslant \int\limits_{A(a, R_1,
R_1+\varepsilon^{\,*}_1)}
\widetilde{Q}(y)\,\eta^n_0(|y-a|)\,dm(y)=\frac{\omega_{n-1}}{I_0^{n-1}}\,,\end{equation}
where $\omega_{n-1}$ is the area of the unit sphere in ${\Bbb R}^n.$
On the other hand, similarly to~(\ref{eq1A}),
$$
M(\Gamma_m)=M(\Gamma(P_m, Q_m, D))\rightarrow\infty\,,\quad
m\rightarrow\infty,
 $$
which contradicts the relation~(\ref{eq7BB}). This completes the
consideration of {\bf Case 1.3}, and therefore, in general, case~1.

{\bf Case~2: One of the points $\overline{x_1}$ or $\overline{x_2}$
is finite, and one is equal to infinity.} Let, for example,
$\overline{x_1}\ne\infty,\,\,\overline{x_2}=\infty.$ Let
$r_m=r_m(t)=(f_m(x_0)-\overline{x_1})t+\overline{x_1},$
$-\infty<t<\infty,$ be a straight line joining the points
$\overline{x_1}$ and $f_m(x_0).$ Join the points $a$ and
$\overline{x_1}$ (even if $a=\overline{x_1}$) by the path $\gamma_1
\colon [1/2, 1] \rightarrow {\Bbb R}^n.$ Let $R_1>0.$ We may assume
that $f_m(x_m)\in B(\overline{x_1}, R_1)$ and $f_m(x_0)\in {\Bbb
R}^n\setminus B(\overline{x_1}, R_1)$ (see Figure~\ref{figure5}).
 \begin{figure}
\centerline{\includegraphics[scale=0.5]{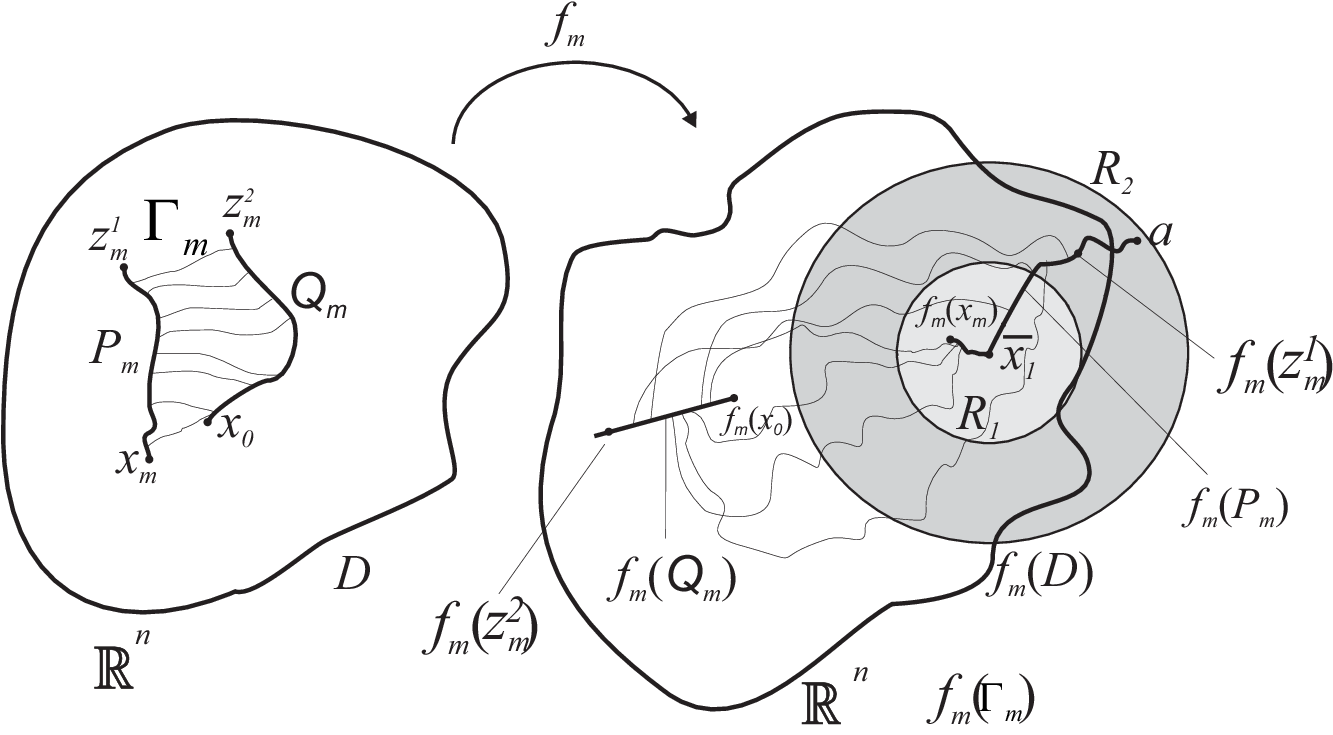}} \caption{To
the proof of Theorem~\ref{th3}, Case~2}\label{figure5}
 \end{figure}
Join the points $f_m(x_m)$ and $\overline{x_1}$ by a path
$\alpha^{\,*}_m\colon [0, 1/2] \rightarrow B(\overline {x_1}, R_1).$
Set
$$
\alpha_m(t)=\quad\left\{\begin{array}{rr}
\alpha^*_m(t), & t\in [0, 1/2],\\
\gamma_1(t), & t\in [1/2, 1]\end{array} \right.\,,\quad
A_1:=|\gamma_1|\cup \overline{B(\overline{x_1}, R_1)}\,.
$$
Now, we establish that there is $R_2>R_1$ and $m_0\in {\Bbb N}$ such
that
\begin{equation}\label{eq3}
A_1\subset B(\overline{x_1}, R_2)\,,\qquad r_m(t)\in {\Bbb
R}^n\setminus B(\overline{x_1}, R_2)\qquad \forall\,\, t\in
[1,\infty)\,,\qquad m\geqslant m_0\,.
\end{equation}
The first inclusion in~(\ref{eq3}) is obvious, because $A_1$ is a
compactum in ${\Bbb R}^n.$ Regarding the second relation, let $x\in
A_1$ and $t\geqslant 1.$ Now, by the triangle inequality,
\begin{gather}
|x-r_m(t)|=|x-((f_m(x_0)-\overline{x_1})t+\overline{x_1})|\nonumber=\\=
\label{eq4}|(f_m(x_0)-\overline{x_1})t+\overline{x_1}-x|\geqslant
|(f_m(x_0)-\overline{x_1})t|-|x-\overline{x_1}|\geqslant\\
\geqslant |f_m(x_0)|-(|\overline{x_1}|+|x-\overline{x_1}|)\geqslant
|f_m(x_0)|-(|\overline{x_1}|+d(A_1))\,.\nonumber\end{gather}
Since $A_1$ is a compactum, there is $c>0$ such that $|x|\leqslant
c$ for any $x\in A_1.$ Now, by~(\ref{eq4}), for $x\in A_1,$
\begin{gather}\label{eq5C}
|r_m(t)|\geqslant |r_m(t)-x|-|x|\geqslant
|f_m(x_0)|-(|\overline{x_1}|+d(A_1))-c\geqslant R_2>R_1>0
\end{gather}
for sufficiently large $m\in {\Bbb N},$ because $f_m(x_0)\rightarrow
\infty$ as $m\rightarrow\infty.$ The latter proves the second
relation in~(\ref{eq3}). Moreover, due to~(\ref{eq5C}) we may
consider that there exists $\varepsilon^{\,*}_1>0$ such that
$r_m(t)\in {\Bbb R}^n\setminus B(\overline{x_1},
R_2+\varepsilon^{\,*}_1)$ for the same $t$ and $m.$

\medskip
Let $x\in A_1$ and $y\in |r_m(t)|,$ $t\geqslant 1,$ $m\geqslant
m_0.$ Now, by~(\ref{eq3}), due to the remarks mentioned above and by
the triangle inequality we obtain that
\begin{gather*}|x-y|=|x-\overline{x_1}+\overline{x_1}-y|\geqslant
\\ \geqslant |y-\overline{x_1}|-|x-\overline{x_1}|\geqslant
R_2+\varepsilon^{\,*}_1-R_2=\varepsilon^{\,*}_1>0\,.
\end{gather*}
The latter means that
\begin{equation}\label{eq1JB}
d(A_1, |r_m(t)|)\geqslant \varepsilon^{\,*}_1>0\,,\qquad  t\in
[1,\infty)\,,\qquad m\geqslant m_0\,.
 \end{equation}
Let $r_0=r_0(y)>0$ be the number from the conditions of the theorem,
defined for each $y_0\in {\Bbb R}^n.$ Set
$$r_*(y):=\min\{\varepsilon^{\,*}_1, r_0(y)\}\,.$$
Cover the set $A_1$ with balls $B(y, r_*/4),$ $y\in A_1.$ Note that
$|\gamma_1 |$ is a compact set in ${\Bbb R}^n$ as a continuous image
of the compact set $[1/2, 1]$ under the mapping $\gamma_1.$ Then, by
the Heine-Borel-Lebesgue lemma, there is a finite subcover
$\bigcup\limits_{i=1}^pB(y_i, r_*/4)$ of the set $A_1.$ In other
words,
\begin{equation}\label{eq2C}
A_1\subset \bigcup\limits_{i=1}^pB(y_i,r_i/4)\,,\qquad 1\leqslant
p<\infty\,,
 \end{equation}
where $r_i$ denotes $r_*(y_i)$ for any $y_i\in {\Bbb R}^n.$

\medskip
Let $\alpha^0_m:[0, c_1)\rightarrow D$ and $r^0_m:[1,
c_2)\rightarrow D$ be maximal $f_m$-liftings of paths $\alpha_m$ and
$r_m$ starting at points $x_m$ and $x_0,$ respectively. Such maximal
liftings exist by Proposition~\ref{pr3}. By the same Proposition,
$\alpha^0_m(t_k)\rightarrow
\partial D$ and $r^0_m(t^{\,\prime}_k)\rightarrow
\partial D$ for some sequences $t_k\rightarrow c_1-0$
and $t^{\,\prime}_k\rightarrow c_2-0,$ $k\rightarrow\infty.$ Then
there are sequences of points $z^1_m\in |\alpha_m^0|$ and $z^2_m\in
|r_m^0|$ such that $d(z^1_m,
\partial D)<1/m$ and $d(z^2_m,\partial D)<1/m.$
Since $\overline{D}$ is a compactum, we may consider that
$z^1_m\rightarrow p_1\in
\partial D$ and $z^2_m\rightarrow p_2\in
\partial D$ as $m\rightarrow\infty.$
Let $P_m$ be the part of the locus of the path $\alpha^0_m$ in $D,$
located between the points $x_m$ and $z^1_m,$ and $Q_m$ the part of
the locus of the path $r^0_m$ in $D,$ located between the points
$x_0$ and $z^2_m.$ By the construction, $f_m(P_m)\subset A_1$ and
$f_m(Q_m)\subset |r_m(t)|_{[1, \infty)}|.$ Put
$\Gamma_m:=\Gamma(P_m, Q_m, D).$ Then, by~(\ref{eq1JB}) and
(\ref{eq2C}), and by~\cite[Theorem~1.I.5.46]{Ku}, we obtain that
\begin{equation}\label{eq5D}
\Gamma_m>\bigcup\limits_{i=1}^p\Gamma_{im}\,,
 \end{equation}
where $\Gamma_{im}:=\Gamma_{f_m}(y_i, r_i/4, r_i/2).$ Set
$\widetilde{Q}(y)=\max\{Q(y), 1\}$ and
$$\widetilde{q}_{y_i}(r)=\frac{1}{\omega_{n-1}r^{n-1}}\int\limits_{S(y_i,
r)}\widetilde{Q}(y)\,d\mathcal{A}\,,$$
where $\omega_{n-1}$ denotes the area of the unit sphere in ${\Bbb
R}^n,$ and $d\mathcal{A}=d\mathcal{H}^{n-1}$ is the area element on
$S(y_i, r).$ Also we set
$$I_i=I_i(y_i,r_i/4,r_i/2)=\int\limits_{r_i/4}^{r_i/2}\
\frac{dr}{r\widetilde{q}_{y_i}^{\frac{1}{n-1}}(r)}\,.$$ Arguing
similarly to the case~1) considered above and taking into account
the relation~(\ref{eq5D}), we obtain that
\begin{equation}\label{eq7E}
M(\Gamma_m)\leqslant \sum\limits_{i=1}^pM(\Gamma_{im})\leqslant
\sum\limits_{i=1}^p\frac{\omega_{n-1}}{I_i^{n-1}}:=C_0\,, \quad
m=1,2,\ldots\,.
\end{equation}
Again, arguing similarly to the case~1) mentioned above, we obtain
that
\begin{equation*}
M(\Gamma_m)=M(\Gamma(P_m, Q_m, D))\rightarrow\infty\,,\quad
m\rightarrow\infty,
\end{equation*}
which contradicts the relation~(\ref{eq7E}). The resulting
contradiction indicates that the assumption in~(\ref{eq13A}) is
wrong, which completes the {\bf Case 2}. Theorem is totally proved.
\end{proof}~$\Box$

\medskip
\begin{corollary}\label{cor3}
{\it\, The statement of Theorem~\ref{th3} is true if the condition
on the function~$Q$ mentioned in this theorem is replaced by a
simpler one: $Q\in L_{\rm loc}^{1}({\Bbb R}^n).$}
\end{corollary}

\medskip
\begin{proof}
By the Fubini theorem (see, e.g., \cite[Theorem~8.1.III]{Sa}) we
obtain that
$$\int\limits_{r_1<|x-x_0|<r_2}Q(x)\,dm(x)=\int\limits_{r_1}^{r_2}
\int\limits_{S(x_0, r)}Q(x)\,d\mathcal{H}^{n-1}(x)dr<\infty\,.$$
This means the fulfillment of the condition of the integrability of
the function $Q$ on the spheres with respect to any subset $E_1$ in
$[r_1, r_2].$
\end{proof}

\section{Proof of Theorem~\ref{th2}}

We will now demonstrate the validity of Theorem~\ref{th2}, i.e., we
will verify that even if a family of mappings omits only one point,
this is sufficient for its equicontinuity. We will prove the theorem
by contradiction.

\medskip
Then there is $\varepsilon_0>0,$ for which the following condition
is true: for any $m\in{\Bbb N}$ there is $x_m \in D$ with
$|x_m-x_0|<1/m,$ and a mapping $f_m \in \frak{F}_{Q, a}(D)$ such
that
\begin{equation}\label{eq13B}
h(f_m(x_m), f_m(x_0))\geqslant \varepsilon_0.
\end{equation}
We may consider that $D$ is bounded. Without loss of generality, we
may assume that $\overline{B(x_0, 1)}\subset D.$ Observe that there
exists $x^{\,\prime}_1\in B(x_0, 1)$ and $m_1>1$ such that
$|f_{m_1}(x^{\,\prime}_1)|>1.$ Otherwise, $|f_m(x)|\leqslant 1$ for
any $B(x_0, 1)$ and any $m\in {\Bbb N}.$ Thus, by Theorem~\ref{th3}
the family $f_m$ is equicontinuous in $B(x_0, 1),$ but this
contradicts with~(\ref{eq13B}). Follow, observe that there exists
$x^{\,\prime}_2\in B(x_0, 1/2)$ and $m_2>m_1$ such that
$|f_{m_2}(x^{\,\prime}_2)|>2.$ Otherwise, $|f_m(x)|\leqslant 2$ for
any $B(x_0, 1/2)$ and any $m\in {\Bbb N},$ $m>m_1.$ Thus, by
Theorem~\ref{th3} the family $f_m$ is equicontinuous in $B(x_0,
1/2),$ but this contradicts with~(\ref{eq13B}). And so on. In
general, observe that there exists $x^{\,\prime}_k\in B(x_0, 1/k)$
and $m_k>m_{k-1}$ such that $|f_{m_{k}}(x^{\,\prime}_k)|>k.$
Otherwise, $|f_m(x)|\leqslant k$ for any $B(x_0, 1/k)$ and any $m\in
{\Bbb N},$ $m_k>m_{k-1}.$ Thus, by Theorem~\ref{th3} the family
$f_m$ is equicontinuous in $B(x_0, 1/k),$ but this contradicts
with~(\ref{eq13B}).

\medskip
So, we have constructed the sequences $x^{\,\prime}_k$ and $f_{m_k}$
$k=1,2,\ldots ,$ such that $x^{\,\prime}_k\rightarrow x_0$ as
$k\rightarrow\infty$ and $f_{m_k}(x^{\,\prime}_k)\rightarrow\infty$
as $k\rightarrow\infty.$ On the other hand, the
relation~(\ref{eq13B}) implies that
\begin{equation}\label{eq14} h(f_{m_k}(x_{m_k}),
f_{m_k}(x_0))\geqslant \varepsilon_0\,, k=1,2,\ldots \,.
\end{equation}
Due to the compactness of $\overline{{\Bbb R}^n},$ we may consider
that sequences $f_m(x_m)$ and $f_m(x_0)$ converge to
$\overline{x_1}$ and $\overline{x_2}\in \overline{{\Bbb R}^n}$ as
$m\rightarrow\infty,$ respectively. The relation~(\ref{eq14})
implies that at least one of the sequence $f_{m_k}(x_{m_k})$ or
$f_{m_k}(x_0)$ does not converge to $\infty.$ Let, for instance,
$f_{m_k}(x_{m_k})\rightarrow\overline{x_1}$ as $k\rightarrow\infty,$
$\overline{x_1}\ne\infty.$ Referring to redesignations, if
necessary, we will henceforth assume that
$f_{m}(x^{\,\prime}_m)\rightarrow\infty$ as $m\rightarrow\infty$ and
$f_{m}(x_{m})\rightarrow\overline{x_1}$ as $m\rightarrow\infty,$
$\overline{x_1}\ne\infty.$

\medskip
The further proof almost completely repeats the scheme presented in
the proof of case~2 of Theorem~\ref{th2} (we will demonstrate this).
Let
$r_m=r_m(t)=(f_m(x^{\,\prime}_m)-\overline{x_1})t+\overline{x_1},$
$-\infty<t<\infty,$ be a straight line joining the points
$\overline{x_1}$ and $f_m(x^{\,\prime}_m).$ Join the points $a$ and
$\overline{x_1}$ (even if $a=\overline{x_1}$) by a path $\gamma_1
\colon [1/2, 1] \rightarrow {\Bbb R}^n.$ Let $R_1>0.$ We may assume
that $f_m(x_m)\in B(\overline{x_1}, R_1)$ and
$f_m(x^{\,\prime}_m)\in {\Bbb R}^n\setminus B(\overline{x_1}, R_1).$
Join the points $f_m(x_m)$ and $\overline{x_1}$ by a path
$\alpha^{\,*}_m\colon [0, 1/2] \rightarrow B(\overline {x_1}, R_1).$
We set
$$
\alpha_m(t)=\quad\left\{\begin{array}{rr}
\alpha^*_m(t), & t\in [0, 1/2],\\
\gamma_1(t), & t\in [1/2, 1]\end{array} \right.\,,\quad
A_1:=|\gamma_1|\cup \overline{B(\overline{x_1}, R_1)}\,.
$$
Arguing similarly to the proof of the relation~(\ref{eq3}), we may
prove that, for some $R_2>R_1>0$
\begin{equation}\label{eq3A}
A_1\subset B(\overline{x_1}, R_2)\,,\qquad r_m(t)\in {\Bbb
R}^n\setminus B(\overline{x_1}, R_2)\qquad \forall\,\, t\in
[1,\infty)\,,\qquad m\geqslant m_0\,.
\end{equation}
Moreover, we may consider that there exists $\varepsilon^{\,*}_1>0$
such that $r_m(t)\in {\Bbb R}^n\setminus B(\overline{x_1},
R_2+\varepsilon^{\,*}_1)$ for the same $t$ and $m.$ Now, similarly
to~(\ref{eq1JB})
\begin{equation}\label{eq1JC}
d(A_1, |r_m(t)|)\geqslant \varepsilon^{\,*}_1>0\,,\qquad  t\in
[1,\infty)\,,\qquad m\geqslant m_0\,.
 \end{equation}
Let $r_0=r_0(y)>0$ be the number from the conditions of the theorem,
defined for each $y_0\in {\Bbb R}^n.$ Set
$$r_*(y):=\min\{\varepsilon^{\,*}_1, r_0(y)\}\,.$$
Cover the set $A_1$ with balls $B(y, r_*/4),$ $y\in A_1.$ Note that
$|\gamma_1 |$ is a compact set in ${\Bbb R}^n$ as a continuous image
of the compact set $[1/2, 1]$ under the mapping $\gamma_1.$ Then, by
the Heine-Borel-Lebesgue lemma, there is a finite subcover
$\bigcup\limits_{i=1}^pB(y_i, r_*/4)$ of the set $A_1.$ In other
words,
\begin{equation}\label{eq2D}
A_1\subset \bigcup\limits_{i=1}^pB(y_i,r_i/4)\,,\qquad 1\leqslant
p<\infty\,,
 \end{equation}
where $r_i$ denotes $r_*(y_i)$ for any $y_i\in {\Bbb R}^n.$

\medskip
Let $\alpha^0_m:[0, c_1)\rightarrow D$ and $r^0_m:[1,
c_2)\rightarrow D$ be maximal $f_m$-liftings of paths $\alpha_m$ and
$r_m$ starting at points $x_m$ and $x_0,$ respectively. Such maximal
liftings exist by Proposition~\ref{pr3}. By the same Proposition,
$\alpha^0_m(t_k)\rightarrow
\partial D$ and $\beta^0_m(t^{\,\prime}_k)\rightarrow
\partial D$ for some sequences $t_k\rightarrow c_1-0$
and $t^{\,\prime}_k\rightarrow c_2-0,$ $k\rightarrow\infty.$ Then
there are sequences of points $z^1_m\in |\alpha_m^0|$ and $z^2_m\in
|r_m^0|$ such that $d(z^1_m,
\partial D)<1/m$ and $d(z^2_m,\partial D)<1/m.$
Since $\overline{D}$ is a compactum, we may consider that
$z^1_m\rightarrow p_1\in
\partial D$ and $z^2_m\rightarrow p_2\in
\partial D$ as $m\rightarrow\infty.$
Let $P_m$ be the part of the locus of the path $\alpha^0_m$ in $D,$
located between the points $x_m$ and $z^1_m,$ and $Q_m$ the part of
the locus of the path $r^0_m$ in $D,$ located between the points
$x_0$ and $z^2_m.$ By the construction, $f_m(P_m)\subset A_1$ and
$f_m(Q_m)\subset |r_m(t)|_{[1, \infty)}|.$ Put
$\Gamma_m:=\Gamma(P_m, Q_m, D).$ Then, by~(\ref{eq1JC}) and
(\ref{eq2D}), and by~\cite[Theorem~1.I.5.46]{Ku}, we obtain that
\begin{equation}\label{eq5E}
\Gamma_m>\bigcup\limits_{i=1}^p\Gamma_{im}\,,
 \end{equation}
where $\Gamma_{im}:=\Gamma_{f_m}(y_i, r_i/4, r_i/2).$ Set
$\widetilde{Q}(y)=\max\{Q(y), 1\}$ and
$$\widetilde{q}_{y_i}(r)=\frac{1}{\omega_{n-1}r^{n-1}}\int\limits_{S(y_i,
r)}\widetilde{Q}(y)\,d\mathcal{A}\,,$$
where $\omega_{n-1}$ denotes the area of the unit sphere in ${\Bbb
R}^n,$ and $d\mathcal{A}=d\mathcal{H}^{n-1}$ is the area element on
$S(y_i, r).$ Also we set
$$I_i=I_i(y_i,r_i/4,r_i/2)=\int\limits_{r_i/4}^{r_i/2}\
\frac{dr}{r\widetilde{q}_{y_i}^{\frac{1}{n-1}}(r)}\,.$$ Arguing
similarly to the case~1) considered above and taking into account
the relation~(\ref{eq5E}), we obtain that
\begin{equation}\label{eq7F}
M(\Gamma_m)\leqslant \sum\limits_{i=1}^pM(\Gamma_{im})\leqslant
\sum\limits_{i=1}^p\frac{\omega_{n-1}}{I_i^{n-1}}:=C_0\,, \quad
m=1,2,\ldots\,.
\end{equation}
Again, arguing similarly to the case~1) mentioned above, we obtain
that
\begin{equation*}
M(\Gamma_m)=M(\Gamma(P_m, Q_m, D))\rightarrow\infty\,,\quad
m\rightarrow\infty,
\end{equation*}
which contradicts the relation~(\ref{eq7F}). The resulting
contradiction indicates that the assumption in~(\ref{eq14}) is
wrong. Theorem is proved. ~$\Box$

\medskip
\begin{example}\label{eq1}
Let $f_m(x)=mx\,,$ $m=1,2,\ldots ,$ $f_m:{\Bbb R}^n\rightarrow {\Bbb
R}^n.$ Observe that, $f_m$ satisfies the
relations~(\ref{eq2*A})--(\ref{eqA2}) with $Q\equiv 1,$ see e.g.
Remark~\ref{rem1}. Note that, the family $\{f_m\}_{m=1}^{\infty}$ is
not equicontinuous at the origin, and that $f_m\not\in\frak{F}_{1,
a}({\Bbb R}^n)$ for any $a\in {\Bbb R}^n.$ By the same reason,
$\{f_m\}_{m=1}^{\infty}$ does not satisfy the conclusion of
Theorem~\ref{th2}. On the other hand, $f_m\in\frak{F}_{1, 0}({\Bbb
R}^n\setminus\{0\})$ and $\{f_m\}_{m=1}^{\infty}$ is equicontinuous
in ${\Bbb R}^n\setminus\{0\}$ by Theorem~\ref{th2}. The latter
conclusion may be also obtained directly and without of application
Theorem~\ref{th2}. The given example shows that the conditions on
the class of mappings~$\frak{F}_{Q, a}(D)$ in Theorem~\ref{th2}
cannot be improved in the sense that in the definition of this class
it is impossible, generally speaking, to refuse to omitting at least
one point by each mapping.
\end{example}

\medskip
\begin{remark}\label{rem2}
Let ${\Bbb M}^n$ and ${\Bbb M}^n_*$ are Riemannian manifolds of
dimension $n$ with geodesic distances $d$ and $d_*,$ respectively.
We set
$q_{x_0}(r)=\frac{1}{r^{n-1}}\int\limits_{S(x_0,
r)}Q(x)\,d\mathcal{A},$
where $d\mathcal{A}$ is the area element of $S(x_0, r).$

For domains $D\subset {\Bbb M}^n,$ $D_*\subset {\Bbb M}^n_*,$
$n\geqslant 2,$ and a function $Q\colon{\Bbb M}^n_*\rightarrow [0,
\infty],$ $Q(x)\equiv 0$ for $x\not\in D_*,$ denote by ${\frak
R}_Q(D, D_*)$ the family of all open discrete mappings $f\colon
D\rightarrow {\Bbb M}_*^n,$ $f(D)=D_*,$ for which $f$ satisfies the
condition
\begin{equation*}
M(\Gamma_f(y_0, r_1, r_2))\leqslant \int\limits_{A(y_0,r_1,r_2)\cap
f(D)} Q(y)\cdot \eta^n (d_*(y,y_0))\, dv_*(y)\,,
\end{equation*}
at each point $y_0\in D_*,$ where $\eta: (r_1,r_2)\rightarrow
[0,\infty ]$ is arbitrary Lebesgue measurable function with
$ \int\limits_{r_1}^{r_2}\eta(r)\, dr\geqslant 1\,.$ (Here $dv_*(y)$
is the volume measure in ${\Bbb M}^n_*$). The following result holds
(see \cite[Theorem~3]{Sev$_1$}):

\medskip
{\it\, Assume that, $\overline{D}$ and $\overline{D_*}$ a compact
sets in ${\Bbb M}^n$ and ${\Bbb M}^n_*,$ respectively,
$\overline{D}_*\ne{\Bbb M}^n_*$ and, in addition, ${\Bbb M}^n_*$ is
connected. Suppose also that the following condition is satisfied:
for each point $y_0\in \overline{D_*}$ there is $r_0=r_0(y_0)>0$
such that $q_{y_0}(r)<\infty$ for each $r\in (0, r_0).$ Then the
family ${\frak R}_Q(D, D_*)$ is equicontinuous in $D.$ }
\medskip

\medskip Note that the conditions of Theorem~\ref{th2}
are very similar to result mentioned in Remark~\ref{rem2}, see
above. However, Theorem~\ref{th2} places no restrictions on the
mapped domain other than the family of mappings omitting some point.
In the statement mentioned above, the family of mappings acts onto a
fixed domain with a compact closure. Therefore, the result of
Theorem~\ref{th2} is much more general than above statement (we
mean, in the situation of the Euclidean space).
\end{remark}

\medskip
\begin{remark}
Let us note the key difference between {\bf Theorem~D} and
Theorem~\ref{th2}. Firstly, it consists in the fact that the mapped
domain $D^{\,\prime}$ is fixed and bounded. Secondly, the condition
on the function $Q$ is slightly different, since we require it not
locally, in the neighborhood of the point, but globally: the numbers
$r_1$ and $r_2$ must vary from $0$ to a specific number
$r_0:=\sup\limits_{y\in D^{\,\prime}}|y-y_0|,$ which may be greater,
generally speaking, than the number $r_0$ specified in
Theorem~\ref{th2}. We should clarify that for the method by which
{\bf Theorem~D} was proved this circumstance is fundamental, since
for a more local nature of the change of $r_1$ and $r_2$ the
corresponding methodology does not work.
\end{remark}

\medskip
{\it Proof of Corollary~\ref{cor1}} is similar to the proof of
Corollary~\ref{cor3} and follows by Theorem~\ref{th2}.~$\Box$

\section{Proof of Theorem~\ref{th1}}

We may consider that $D$ is a bounded domain. Let us prove this
assertion by contradiction. Assume that the family $\frak{R}_{Q}(D)$
is not equicontinuous at $x_0.$ Then there is $\varepsilon_0>0,$ for
which the following condition is true: for any $m\in{\Bbb N}$ there
is $x_m \in D$ with $|x_m-x_0|<1/m,$ and a mapping $f_m \in
\frak{F}_{Q, a}(D)$ such that
\begin{equation}\label{eq13C}
h(f_m(x_m), f_m(x_0))\geqslant \varepsilon_0.
\end{equation}
Since $\overline{{\Bbb R}^n}$ is a compact space, we may consider
that sequences $f_m(x_m)$ and $f_m(x_0)$ converge to
$\overline{x_1}$ and $\overline{x_2}\in \overline{{\Bbb R}^n}$ as
$m\rightarrow\infty,$ respectively. By~(\ref{eq13C}) and by the
continuity of the metrics, we obtain that $h(\overline{x_1},
\overline{x_2})\geqslant\varepsilon_0/2.$

\medskip
We may assume that $\overline{x_1}\ne \infty$ and that
$\overline{x_1}=y_0,$ where $y_0$ is the element from the conditions
of the theorem. Let $B(\overline{x_1}, R_1)$ be a ball centered at
$\overline{x_1},$ $R_1>0,$ such that $\overline{x_2}\not \in
B(\overline{x_1}, R_1).$ We may assume that $f_m(x_m)\in
B(\overline{x_1}, R_1)$ and $f_m(x_0)\in {\Bbb R}^n\setminus
B(\overline{x_1}, R_1),$
see Figure~\ref{figure6}.
 \begin{figure}
\centerline{\includegraphics[scale=0.5]{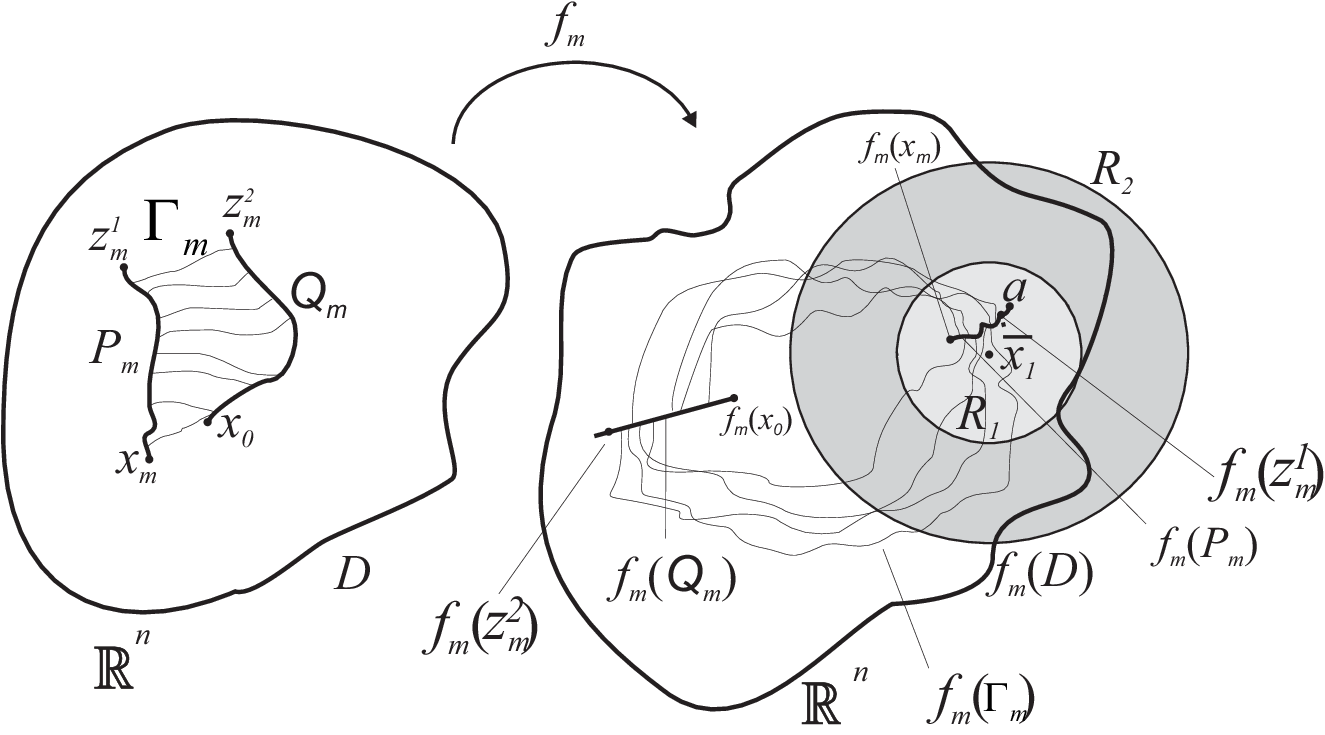}} \caption{To
the proof of Theorem~\ref{th1}}\label{figure6}
 \end{figure}
In addition, we may assume that $0<R_1<r_0(y_0),$ where $r_0(y_0)$
is a number from the formulation of Theorem~\ref{th1}. Let
$a=y_{m_0}\in B(\overline{x_1}, R_1),$ $a\ne y_0,$ be a point,
$m_0\in {\Bbb N},$ such that, if $x\in f^{\,-1}_{m}(a),$
$m=1,2,\ldots ,$ then
\begin{equation}\label{eq6}
|x-x_0|\geqslant C_{m_0}\,, m=1,2,\ldots \,.\end{equation}
The existing of the number $m_0,$ the point $y_{m_0}$ and the
constant $C_{m_0}$ is ensured by the assumptions made in
Theorem~\ref{th1}.

\medskip
The further proof almost completely repeats the scheme presented in
the proof of Theorem~\ref{th2}. Let
$r_m=r_m(t)=(f_m(x_0)-\overline{x_1})t+\overline{x_1},$
$-\infty<t<\infty,$ be a straight line joining the points
$\overline{x_1}$ and $f_m(x_0).$ Join the points $f_m(x_m)$ and $a$
by a path $\alpha_m\colon [0, 1] \rightarrow B(\overline {x_1},
R_1).$
Set $A_1:=\overline{B(\overline{x_1}, R_1)}.$
Arguing similarly to the proof of the relation~(\ref{eq3}), we may
prove that, for some $R_2>R_1>0$
\begin{equation}\label{eq3B}
A_1\subset B(\overline{x_1}, R_2)\,,\qquad r_m(t)\in {\Bbb
R}^n\setminus B(\overline{x_1}, R_2)\qquad \forall\,\, t\in
[1,\infty)\,,\qquad m\geqslant m_1>m_0\,.
\end{equation}
Let $\alpha^0_m:[0, c_1)\rightarrow D$ and $r^0_m:[1,
c_2)\rightarrow D$ be maximal $f_m$-liftings of paths $\alpha_m$ and
$r_m$ starting at points $x_m$ and $x_0,$ respectively. Such maximal
liftings exist by Proposition~\ref{pr3}. By the same Proposition,
$r^0_m(t^{\,\prime}_k)\rightarrow
\partial D$ for some sequence $t^{\,\prime}_k\rightarrow c_2-0.$ Besides that,
if $f^{\,-1}_{m}(a)=\varnothing,$ $m=1,2,\ldots ,$ then
$\alpha^0_m(t_k)\rightarrow
\partial D$ as $k\rightarrow\infty$ (see Proposition~\ref{pr3}).
Otherwise, if $f^{\,-1}_{m}(a)\ne\varnothing$ for some $m=1,2,\ldots
,$ then (by Proposition~\ref{pr3}) there are two cases: 1) either
$\alpha^0_m(t_k)\rightarrow
\partial D$ for some sequence $t_k\rightarrow c_1-0$ or 2) $\alpha^0_m(t)\rightarrow
x_1$ as $t\rightarrow c_1=1,$ $x_1\in D.$ In the second case,
$x_1\in \{f^{\,-1}_{m}(a)\},$ but by~(\ref{eq6}),
$|x_1-x_0|\geqslant C_{m_0}.$ Now, by the triangle inequality,
\begin{gather*}
|\alpha^0_m(0)-\alpha^0_m(1)|=|x_m-x_1|\geqslant\\
\geqslant|x_1-x_0|-|x_0-x_m|\geqslant |x_1-x_0|/2\geqslant C_{m_0}/2
\end{gather*}
for sufficiently large $m\in {\Bbb N}.$ The latter implies that
$d(\alpha^0_m)\geqslant \min \{C_{m_0}/2, d(x_0, \partial D)\},$
$m\geqslant m_2,$ $m_2>m_1.$ Then there are sequences of points
$z^1_m\in |\alpha_m^0|$ and $z^2_m\in |r_m^0|$ and a number
$\delta_0>0$ such that $d(z^1_m, x_0)>\delta_0$ and $d(z^2_m,
x_0)>\delta_0$ for sufficiently large $m\in {\Bbb N}.$
Let $P_m$ be the part of the locus of the path $\alpha^0_m$ in $D,$
located between the points $x_m$ and $z^1_m,$ and $Q_m$ the part of
the locus of the path $r^0_m$ in $D,$ located between the points
$x_0$ and $z^2_m.$ By the construction, $f_m(P_m)\subset A_1$ and
$f_m(Q_m)\subset |r_m(t)|_{[1, \infty)}|.$ Put
$\Gamma_m:=\Gamma(P_m, Q_m, D).$ Then, by Lemma~\ref{lem2C}
\begin{equation}\label{eq8}
M(\Gamma_m)=M(\Gamma(P_m, Q_m, D))\rightarrow\infty\,,\quad
m\rightarrow\infty\,.
\end{equation}
On the other hand,  we set $\widetilde{Q}(y)=\max\{Q(y), 1\}$ and
$$\widetilde{q}_{a}(r)=\frac{1}{\omega_{n-1}r^{n-1}}\int\limits_{S(a,
r)}\widetilde{Q}(y)\,d\mathcal{A}\,,$$
where $\omega_{n-1}$ denotes the area of the unit sphere in ${\Bbb
R}^n,$ and $d\mathcal{A}=d\mathcal{H}^{n-1}$ is the area element on
$S(a, r).$ Also we set
$$I_a=I_a(a, R_1, R_2)=\int\limits_{R_1}^{R_2}\
\frac{dr}{r\widetilde{q}_{a}^{\frac{1}{n-1}}(r)}\,.$$
Then, by~(\ref{eq3B}) and~\cite[Theorem~1.I.5.46]{Ku}, we obtain
that
\begin{equation}\label{eq5F}
\Gamma_m>\Gamma_{f_m}(a, R_1, R_2)\,.
 \end{equation}
Arguing similarly to the case~1) of the proof of Theorem~\ref{th3},
using~(\ref{eq5F}) we obtain that
\begin{equation}\label{eq7G}
M(\Gamma_m)\leqslant M(\Gamma_{f_m}(a, R_1, R_2))\leqslant
\frac{\omega_{n-1}}{I_a^{n-1}}:=C_0\,, \quad m=1,2,\ldots\,.
\end{equation}
The relations~(\ref{eq8}) and~(\ref{eq7G}) contradict each other.
The contradiction obtained above finishes the proof.~$\Box$

\medskip
{\it Proof of Corollary~\ref{cor2}} is similar to the proof of
Corollary~\ref{cor3} and follows by Theorem~\ref{th1}.~$\Box$

\medskip\medskip\medskip
{\bf Funding.} The work was supported by the National Research
Foundation of Ukraine, Project number 2025.07/0014, Project name:
``Modern problems of Mathematical Analysis and Geometric Function
Theory''.

\medskip
{\bf \noindent Miodrag Mateljevi\'{c}} \\ University of Belgrade,
Faculty
of Mathematics \\
16 Studentski trg, P.O. Box 550\\
11001 Belgrade, SERBIA\\
miodrag@matf.bg.ac.rs

\medskip
\medskip
{\bf \noindent Evgeny Sevost'yanov} \\
{\bf 1.} Zhytomyr Ivan Franko State University,  \\
40 Velyka Berdychivska Str., 10 008  Zhytomyr, UKRAINE \\
{\bf 2.} Institute of Applied Mathematics and Mechanics\\
of NAS of Ukraine, \\
19 Henerala Batyuka Str., 84 116 Slov'yansk,  UKRAINE\\
esevostyanov2009@gmail.com

\end{document}